\documentclass[11pt,a4paper,fleqn]{article}

\usepackage{amsmath,amssymb,amsthm,enumerate,cite,tikz}
\usepackage{longtable}

\newcommand{\ev}{\mathop{\mathbf{ev}}\nolimits}
\newcommand{\F}{\mathbb{F}}

\newcommand{\GlnF}{\mathop{{\rm GL}(n,\F)}\nolimits}
\newcommand{\GlF}{\mathop{{\rm GL}(3,\F)}\nolimits}

\newcounter{tbn}
\newcounter{pict}

\newcounter{mcasenum}

\newtheorem{theorem}{Theorem}[section]
\newtheorem{lemma}[theorem]{Lemma}
\newtheorem*{lemma*}{Lemma}
\newtheorem{corollary}[theorem]{Corollary}
\newtheorem*{corollary*}{Corollary}
\newtheorem*{proposition*}{Proposition}

{\theoremstyle{definition} 
\newtheorem*{definition*}{Definition}

\newtheorem*{example*}{Example}
\newtheorem*{examples*}{Exampes}
\newtheorem{remark}[theorem]{Remark}
\newtheorem*{remark*}{Remark}

\newtheorem*{note*}{Note}
\newtheorem{result}[theorem]{Result}
\newtheorem*{result*}{Result}
\newtheorem*{comment*}{Comment}

\begin{document}

\title{\bf Degenerations of 3-dimensional nilpotent associative algebras over an algebraically closed field
}


\author{N. M. Ivanova%
\thanks{
\textit{\noindent Institute of Mathematics of NAS of Ukraine,~3 Tereshchenkivska Str., 01601 Kyiv, Ukraine and 
European University of Cyprus, Nicosia, Cyprus.}
{E-mail: ivanova.nataliya@gmail.com}
}
\ and\ \
C. A. Pallikaros%
\thanks{%
\textit{Department of Mathematics and Statistics,
University of Cyprus,
P.O.Box 20537,
1678 Nicosia,
Cyprus.}
{E-mail: pallikaros.christos@ucy.ac.cy}
}
}
\date{August 16, 2024}

\maketitle

\begin{abstract}
We determine the complete degeneration picture inside the variety of nilpotent associative algebras of dimension 3 over an algebraically closed field.
Comparing with the discussion in~\cite{IvanovaPallikaros2022},
for some of the arguments in the present article we needed to develop alternative techniques which are now valid over an arbitrary algebraically closed field.
There is a dichotomy of cases concerning the results obtained, corresponding to whether the characteristic of the field is 2 or not.
\\[1.5ex]
\end{abstract}



\section{Introduction}

The notion of degeneration or contraction arises in various physical investigations.
It was first introduced by Segal~\cite{Segal1951} and In\"on\"u and Wigner~\cite{InonuWigner1953,InonuWigner1954}.
They showed that certain Lie algebras can be related to one another through a kind of limiting process (degeneration).
The degeneration process demonstrated in~\cite{Segal1951,InonuWigner1953,InonuWigner1954} was useful in the direction of understanding  the relationship between different Lie algebras and their representations.
This provided a bridge between seemingly unrelated mathematical structures and motivated the exploration of connections between different areas of mathematics.

There are many works devoted to determining the degenerations within various classes of algebras, mainly over the fields $\mathbb C$ and $\mathbb R$ --- see, for example,~\cite{Agaoka1999,Burde&Steinhoff1999,NesterenkoPopovych2006} for the cases of 3- and 4-dimensional Lie algebras.
For various investigations related to degenerations of algebras over an arbitrary field see, for example,~\cite{IvanovaPallikaros2019,IvanovaPallikaros2019b,PallikarosWard2020,Ward2023}.
The present paper is concerned with the degenerations of 3-dimensional nilpotent associative algebras over an algebraically closed field.
The classification of 3-dimensional nilpotent associative algebras over an arbitrary field can be found in~\cite{KrusePrice1969} or~\cite{deGraaf2018}.
Based on this classification we determine the complete degeneration picture within this class of algebras over an algebraically closed field.
We thus generalize the part of the degeneration picture obtained in~\cite{IvanovaPallikaros2022} concerning nilpotent associative algebras over the field of complex numbers.
In this paper, we needed to employ alternative techniques, especially for degenerations to and from the infinite family in the above classification, which are now valid over an arbitrary algebraically closed field and not just over~$\mathbb C$.

It turns out that there is a natural split between the cases ${\rm char}\,\F\ne2$ and ${\rm char}\,\F=2$, giving rise to different degeneration pictures for these two cases.

\section{Preliminaries}\label{sec:Prilim}

We suppose that $\F$ is an arbitrary infinite field and $n$ a positive integer.
We also let $V$ be an $\F$-vector space with $\dim_\F V=n$ and fix an ordered $\F$-basis $(e_1,\ldots,e_n)$ of $V$ which we will refer to as the standard basis of $V$.

An algebra structure $\mathfrak g$ on $V$ is an $\F$-algebra having $V$ as its underlying vector space, in particular the multiplication in $\mathfrak g$ is defined via a suitable $\F$-bilinear map $[,]_{\mathfrak g}\colon V\times V\to V$.
We denote by $\boldsymbol A$ the set of all algebra structures on $V$ and, for $\mathfrak g, \mathfrak h\in \boldsymbol A$, by $\mathfrak g\simeq\mathfrak h$ we mean that $\mathfrak g$ and $\mathfrak h$ are isomorphic as $\F$-algebras.
Writing, for simplicity, $xy$ in the place of $[x,y]_{\mathfrak g}$ for $x,y\in V$, an algebra $\mathfrak g\in\boldsymbol A$ is called associative if $(uv)w=u(vw)$ for all $u,v,w\in V$.
It is called nilpotent if, for some positive integer $r$, the product $(${}{\tiny$\cdots$}$((w_1w_2)w_3)\ldots)w_{r+1}=0_V$ for all $w_1,\ldots,w_{r+1}\in V$ (in this article we call the least such $r$ the nilpotency class of the algebra).

If now $(u_1,\ldots,u_n)$ is any ordered $\F$-basis of $V$, the multiplication in $\mathfrak g\in\boldsymbol A$ is completely determined by the structure constants $\gamma_{ijk}\in\F$ ($1\le i,j,k\le n$), given by $[u_i,u_j]_{\mathfrak g}=\sum_{k=1}^n\gamma_{ijk}u_k$.
It will be convenient to regard this set of structure constants $\gamma_{ijk}$ as an ordered $n^3$-tuple (via a suitable ordering of the ordered triples $(i,j,k)$ for $1\le i,j,k\le n$).
In such a case we call the ordered $n^3$-tuple $\boldsymbol\gamma=(\gamma_{ijk})$ the structure vector of $\mathfrak g\in\boldsymbol A$ relative to the basis $(u_1,\ldots,u_n)$ of $V$.
We denote by $\boldsymbol\Lambda$ the $\F$-vector space formed, via the usual (componentwise) addition and scalar multiplication, by all possible structure vectors of elements of $\boldsymbol A$.
Clearly, $\boldsymbol\Lambda=\F^{n^3}$ as sets.
Moreover, an element $\boldsymbol\lambda\in\boldsymbol\Lambda$ occurs as a structure vector for both algebras $\mathfrak g, \mathfrak h$ in $\boldsymbol A$ exactly when $\mathfrak g\simeq\mathfrak h$.
We can obtain a natural bijection $\Theta\colon\boldsymbol A\to\boldsymbol\Lambda$ by defining, for $\mathfrak g\in\boldsymbol A$, the element $\Theta(\mathfrak g)$ of $\boldsymbol\Lambda$ to be the structure vector of $\mathfrak g$ relative to the standard basis $(e_1,\ldots, e_n)$ of $V$ we have fixed above.
We will use symbol ${\bf abc}$ for the element $\boldsymbol\lambda\,(=(\lambda_{ijk}))$ of $\boldsymbol\Lambda$ with $\lambda_{abc}=1_\F$ and with all other $\lambda_{ijk}$ equal to $0_\F$.
We will refer to the basis $\{{\bf ijk}\colon 1\le i,j,k\le n\}$ of $\boldsymbol\Lambda$ as the standard basis of~$\boldsymbol\Lambda$.

We will be working with the Zariski topology on the set $\boldsymbol\Lambda$, so at this point we recall some basic facts on algebraic sets.
Let $\F[{\bf X}]$ be the ring $\F[X_{ijk}: 1\le i,j,k\le n]$ of polynomials in the indeterminates $X_{ijk}$ ($1\le i,j,k\le n$) over $\F$.
A subset $W$ of $\boldsymbol\Lambda$ is called algebraic (or Zariski-closed) if there exists a subset $S\subseteq \F[{\bf X}]$ such that $W=\{\boldsymbol\lambda=(\lambda_{ijk})\in\boldsymbol\Lambda\colon \ev_{\boldsymbol\lambda}(f)=0_{\F}$ for all $f\in S\}$ where, for $\boldsymbol\lambda=(\lambda_{ijk})\in\boldsymbol\Lambda$, we denote by $\ev_{\boldsymbol\lambda}$ the evaluation homomorphism $\F[{\bf X}]\to\F$ (which is the identity on $\F$ and sends $X_{ijk}$ to $\lambda_{ijk}$ for $1\le i,j,k\le n$).
The vanishing ideal of a subset $U$ of $\boldsymbol\Lambda$ is defined by ${\bf I}(U)=\{f\in\F[{\bf X}]\colon \ev_{\boldsymbol\lambda}(f)=0_{\F}$ for all $\boldsymbol\lambda\in U\}$.
The Zariski closure of $U$ will be denoted by $\overline{U}$.
Also note that in a similar way we can define algebraic subsets of $\F^r$ with $r$ an arbitrary positive integer.

Let $G=\GlnF$ be the general linear group of degree $n$ over $\F$, that is $G=\{g\,(=(g_{ij}))\in\F^{n^2}\colon \det g\ne0_\F\}$.
We can identify $G$ with the Zariski-closed subset of points $(g_{ij},b)$ of $\F^{n^2+1}$ satisfying $b\det(g_{ij})=1_\F$, and thus consider $G$ as an affine algebraic group.
Thus, the algebra of regular functions on $G$ is the $\F$-algebra $\F[t_{ij}\, (1\le i,j\le n),\, d^{-1}]$ of polynomial maps $G\to\F$, where the $t_{ij}$ are the coordinate functions on $G$ (that is, for $g=(g_{ij})\in G$ we have $t_{ij}(g)=g_{ij}$), and $d=\det(t_{ij})$.

There is a natural linear action of $G$ on $\boldsymbol\Lambda$ ``by change of basis'' which gives $\boldsymbol\Lambda$ the structure of an algebraic $G$-module:
Consider the map $\phi\colon\boldsymbol\Lambda\times G\to\boldsymbol\Lambda:$ $(\boldsymbol\lambda,g)\mapsto\boldsymbol\lambda g$, ($\boldsymbol\lambda\in\boldsymbol\Lambda$, $g=(g_{ij})\in G$), where we define $\boldsymbol\lambda g\in\boldsymbol\Lambda$ to be the structure vector of $\Theta^{-1}(\boldsymbol\lambda)\in\boldsymbol A$ relative to the ordered $\F$-basis $(v_1,\ldots, v_n)$ of $V$ given by $v_j=\sum_{i=1}^n g_{ij}e_i$ for $1\le j\le n$ (in other words $g\in G$ is taken to be the transition matrix from the basis $(e_i)_{i=1}^n$ to the basis $(v_i)_{i=1}^n$ of $V$).
It is easy to check that $\phi$ indeed defines a linear (right) action of $G$ on $\boldsymbol\Lambda$ satisfying ${\bf ijk}\,g=\sum_{a,b,c}g_{ia}g_{jb}\hat g_{ck}\,{\bf abc}$, where we denote by $\hat g_{ij}$ the $(i,j)$-entry of matrix $g^{-1}$.
In fact, if we use the ``lexicographic'' ordering for the standard basis elements ${\bf ijk}$ of $\boldsymbol\Lambda$ and the natural identification between the elements of $\boldsymbol\Lambda$ and $(1\times n^3)$-matrices with entries in $\F$, we see that, for each $g\in G$, the map $\phi_g\colon\boldsymbol\Lambda\to\boldsymbol\Lambda:$ $\boldsymbol\lambda\mapsto\boldsymbol\lambda g$ is the $\F$-linear map given by post-multiplication of $\boldsymbol\lambda$ by the matrix $g\otimes(g\otimes(g^{-1})^{\rm tr})$.
Since the coefficients of this last matrix induce regular functions on $G$, this ensures that we can regard $\boldsymbol\Lambda$ as a $G$-variety via the action defined by $\phi$ (compare with the discussion in~\cite[Sections~3.1, 3.2]{Geck2003}).

Given $\boldsymbol\lambda\in\boldsymbol\Lambda$, we will be considering the $H$-orbit $\boldsymbol\lambda H\,(=\{\boldsymbol\lambda g\colon g\in H\})$ for some particular subgroups $H$ of $G$, including $G$ itself.
We will write $\boldsymbol\lambda G=O(\boldsymbol\lambda)$.
Since, for each $g\in G$, the map $\phi_g$ from $\boldsymbol\Lambda$ to $\boldsymbol\Lambda$ sending $\boldsymbol\lambda$ to $\boldsymbol\lambda g$ is regular, and hence continuous in the Zariski topology, we get (compare, for example, with~\cite[Lemma~3.1]{IvanovaPallikaros2019} and its proof):

\begin{result}\label{Result2.1}
Let $\boldsymbol\lambda\in\boldsymbol\Lambda$ and let $H$ be a subgroup of $G$.
Then $\overline{\boldsymbol\lambda H}$ is a union of $H$-orbits (and hence a $H$-invariant subset of $\boldsymbol\Lambda$).
\end{result}

In particular, if $\boldsymbol\lambda,\boldsymbol\mu\in\boldsymbol\Lambda$ with $\boldsymbol\mu\in O(\boldsymbol\lambda)$ then $O(\boldsymbol\mu)\subseteq\overline{O(\boldsymbol\lambda)}$.
This, together with the easy observation that the $G$-orbits of the action defined by $\phi$ correspond, via the map $\Theta$, precisely to the isomorphism classes of $n$-dimensional $\F$-algebras, lead to the following definition.

\begin{definition*}
Let $\mathfrak g,\mathfrak h\in\boldsymbol A$.
We say that $\mathfrak g$ degenerates to $\mathfrak{h}$ (respectively, $\mathfrak{g}$ properly degenerates to $\mathfrak{h}$) if $\Theta(\mathfrak h) \in\overline{O(\Theta(\mathfrak g) )}$
(respectively, $\Theta(\mathfrak h) \in\overline{O(\Theta(\mathfrak g) )}- O(\Theta(\mathfrak g) )$).
For $\boldsymbol\lambda,\boldsymbol\mu\in\boldsymbol\Lambda$ we  write $\boldsymbol\lambda\to\boldsymbol\mu$ (or $\Theta^{-1}(\boldsymbol\lambda)\to\Theta^{-1}(\boldsymbol\mu)$) if $\boldsymbol\mu\in\overline{O(\boldsymbol\lambda)}$.
Otherwise we write $\boldsymbol\lambda\not\to\boldsymbol\mu$ (or $\Theta^{-1}(\boldsymbol\lambda)\not\to\Theta^{-1}(\boldsymbol\mu)$).
\end{definition*}

\begin{result}\label{Result2.2}
Let $\boldsymbol\lambda,\boldsymbol\mu,\boldsymbol\nu\in\boldsymbol\Lambda$.

(i) If $\boldsymbol\lambda\to\boldsymbol\mu$ and $\boldsymbol\mu\to\boldsymbol\nu$, then $\boldsymbol\lambda\to\boldsymbol\nu$.

(ii) If $\boldsymbol\lambda\to\boldsymbol\mu$ and, in addition, $\F$ is algebraically closed, then $\boldsymbol\mu\not\to\boldsymbol\lambda$.
\end{result}
\begin{proof}
For item (i) it is enough to observe that from $\boldsymbol\mu\in\overline{O(\boldsymbol\lambda)}$ we get $O(\boldsymbol\mu)\subseteq\overline{O(\boldsymbol\lambda)}$ (see Result~\ref{Result2.1}), and hence $\overline{O(\boldsymbol\mu)}\subseteq\overline{O(\boldsymbol\lambda)}$.
For a proof of item (ii) see~\cite[Proposition~2.5.2]{Geck2003}.
\end{proof}

We denote by $B$ the Borel subgroup of $G$ consisting of all upper triangular matrices in $G$.
Then $B$ is a parabolic subgroup of $G$.
If $\boldsymbol\lambda\in\boldsymbol\Lambda$ and $\F$ is algebraically closed we have, see for example~\cite[Proposition~3.2.12]{Geck2003}, that the set $(\overline{\boldsymbol\lambda B})G\,(=\{xg\colon x\in\overline{\boldsymbol\lambda B}, g\in G\})$ is a Zariski-closed subset of $\boldsymbol\Lambda$.
Hence we get (see, for example,~\cite[Lemma~2.4]{IvanovaPallikaros2022} for the details of the proof):

\begin{result}\label{Result2.3}
Suppose $\F$ is algebraically closed.
If $\boldsymbol\lambda\in\boldsymbol\Lambda$, then $(\overline{\boldsymbol\lambda B})G=\overline{\boldsymbol\lambda G}\, (=O(\boldsymbol\lambda))$.
\end{result}

Finally for this section, we recall one more result from~\cite{IvanovaPallikaros2022} which will be useful later on.

\begin{result}\label{Result1}
(See \cite[Lemma~2.2]{IvanovaPallikaros2022}.)
Let $f\colon\F\to\boldsymbol\Lambda$ be a continuous function in the Zariski topology and let $S=\F-\{0_{\F}\}$.
Then $f(0_\F)\in\overline{f(S)}$.
\end{result}

{\bf Notation.}
For the rest of the paper we fix $n=3$, and restrict to the case the field $\F$ is algebraically closed.
Symbol $\mathcal T$ will denote the set of the $27\,(=n^3)$ ordered triples $\{(i,j,k)\colon 1\le i,j,k\le 3\}$.
For $m\in\mathbb Z-\{0\}$, by the symbol $m\cdot1_\F$ we mean the sum of the $|m|$ terms $1_\F+\ldots+1_\F$ (resp., $-1_\F+\ldots+-1_\F$) if $m>0$ (resp., $m<0$).
We fix $\mathcal N$ to be the subset of $\boldsymbol\Lambda$ defined by $\Theta^{-1}(\mathcal N)=\{\mathfrak g\in\boldsymbol A\colon \mathfrak g$ is nilpotent associative$\}$.
Finally, recall that by $(e_1,e_2,e_3)$ we denote the standard (ordered) basis we have fixed for the underlying vector space $V$ on which all our algebras are defined,
and by $[,]_{\mathfrak g}$ we mean the product of two elements of such an algebra $\mathfrak g\in\boldsymbol A$ --- this notation should not be confused with commutators or Lie products.

\section{Two useful lemmas}\label{SectionTwoLemmas}

For each $\beta\in\F$ we define the element $\boldsymbol\sigma(\beta)\in\boldsymbol\Lambda$ by $\boldsymbol\sigma(\beta)={\bf231}+\beta{\bf321}$.
The algebras $\Theta^{-1}(\boldsymbol\sigma(\beta))$ are then associative and also nilpotent of class 2.
Moreover, for $\beta\in\F-\{-1_\F\}$ the dimension of the Lie algebra of derivations is the same for all these algebras.
As here we are working over an arbitrary algebraically closed field (of characteristic not necessarily equal to zero), it is therefore not possible for us (at least with the general results we have at our disposal at the moment) to rule out the possibility of degeneration between  members of this family (for $\beta\ne-1_\F$), as one would be justified to do when considering the special case $\F=\mathbb C$.
(See, for example~\cite[Proposition~3.2(iv)]{IvanovaPallikaros2022}, in particular the references included in its proof for an explanation.)
In the following two lemmas we employ an alternative technique which helps us overcome this extra difficulty which did not arise in~\cite{IvanovaPallikaros2022} where we worked over~$\mathbb C$.

\medskip
It will also be convenient to define an equivalence relation~$\mathcal R$ on $\F$ by $\mathcal R=\{(\alpha,\beta)\in\F\times\F\colon \alpha=\beta$ or $\alpha\beta=1_\F\}$.
When ${\rm char}\,\F\ne2$, the equivalence classes corresponding to the equivalence relation $\mathcal R$ are $\overline{0_\F}=\{0_\F\}$, $\overline{1_\F}=\{1_\F\}$, $\overline{-1_\F}=\{-1_\F\}$ and $\overline{\alpha}=\{\alpha,\alpha^{-1}\}$ for $\alpha\in\F-\{0_\F,1_\F,-1_\F\}$, with the obvious adjustments when ${\rm char}\,\F=2$.
[We denote by $\overline{\beta}$ the equivalence class of $\beta\in\F$ relative to equivalence relation $\mathcal R$.]


%

\begin{lemma}\label{Lemma_n4_are_non-isomorphic}
Let $\beta,\xi\in\F$ with $\beta\not\in\overline{\xi}$.
Then $\boldsymbol\sigma(\xi)\not\in\overline{O(\boldsymbol\sigma(\beta))}$.
(In particular, $\boldsymbol\sigma(\xi)\not\in O(\boldsymbol\sigma(\beta))$ whenever $\beta\not\in\overline{\xi}$.)
\end{lemma}


\begin{proof}
Assume the hypothesis.
Set $\boldsymbol\lambda=\boldsymbol\sigma(\beta)$ and $\boldsymbol\nu=\boldsymbol\sigma(\xi)$.
In view of Result~\ref{Result2.3} it suffices to show that $\overline{\boldsymbol\lambda B}\cap O(\boldsymbol\nu)=\varnothing$.
Suppose, on the contrary, that $\overline{\boldsymbol\lambda B}\cap O(\boldsymbol\nu)\ne\varnothing$ and let $\boldsymbol\mu\,(=(\mu_{ijk}))\in\overline{\boldsymbol\lambda B}\cap O(\boldsymbol\nu)$.
Then $\boldsymbol\mu=\boldsymbol\nu g$ for some $g\,(=(g_{ij}))\in G\,(=\GlF)$.
Defining $\gamma=(\det g)^{-1}(g_{22}g_{33}-g_{23}g_{32})\in\F$, an easy computation gives the following relations
\begin{gather*}
\mu_{231}=\gamma(g_{22}g_{33}+\xi g_{23}g_{32}),\qquad \mu_{321}=\gamma(g_{23}g_{32}+\xi g_{22}g_{33}),\\
\mu_{331}=\gamma(1_\F+\xi) g_{23}g_{33},\qquad \mu_{221}=\gamma(1_\F+\xi) g_{22}g_{32}.
\end{gather*}
Also observe that for $b\,(=(b_{ij}))\in B$ we have that
\[
\boldsymbol\lambda b=\Big(\frac{b_{22}b_{33}}{b_{11}}\Big){\bf231}+\beta\Big(\frac{b_{22}b_{33}}{b_{11}}\Big){\bf321}+(1_\F+\beta)\Big(\frac{b_{23}b_{33}}{b_{11}}\Big){\bf331}, \mbox{\ with\ } b_{11}b_{22}b_{33}\ne0_\F
\]
(compare with~\cite[Example~2.6]{IvanovaPallikaros2022}).
It follows that the polynomials $X_{321}-\beta X_{231}$ and $X_{ijk}$ for all $(i,j,k)\in \mathcal T-\{(2,3,1),\, (3,2,1),\, (3,3,1)\}$ belong to the vanishing ideal ${\bf I}(\boldsymbol\lambda B)$.
This leads to $\mu_{321}=\beta\mu_{231}$ and $\mu_{ijk}=0_\F$ whenever the ordered triple $(i,j,k)$ belongs to the set $\mathcal T-\{(2,3,1),\, (3,2,1),\, (3,3,1)\}$.
If $\gamma=0_\F$ (equivalently, $g_{22}g_{33}-g_{23}g_{32}=0_\F$) we see immediately that  $\mu_{231}=\mu_{321}=\mu_{331}=0_\F$ leading to $\boldsymbol\mu={\bf0}$, a contradiction since ${\bf0}\not\in O(\boldsymbol\nu)$.
So, for the rest of the proof we assume that $\gamma\ne0_\F$ and consider the two subcases below.

(i) We suppose first that $\xi=-1_\F$.
It follows that $\mu_{321}+\mu_{231}=0_\F$.
Invoking the fact that $\mu_{321}=\beta\mu_{231}$ (see discussion above), we get $(1_\F+\beta)\mu_{231}=0_\F$.
But $1_\F+\beta\ne0_\F$ since $\beta\not\in\overline{\xi}\,(=\{-1_\F\})$ by assumption.
Hence $\mu_{231}=0_\F$ giving $\mu_{321}=0_\F$ also.
We conclude that $\boldsymbol\mu={\bf0}$ in this case since the assumption $\xi=-1_\F$ also ensures that $\mu_{331}=0_\F$.

(ii) Finally, we suppose that $\xi\ne-1_\F$.
From the assumption $\gamma\ne0_\F$ and the condition $\mu_{321}=\beta\mu_{231}$ obtained above we get that $g_{23}g_{32}+\xi g_{22}g_{33}=\beta g_{22}g_{33}+(\beta\xi)g_{23}g_{32}$.
Hence  $(1_\F-\beta\xi)g_{23}g_{32}=(\beta-\xi)g_{22}g_{33}$.
It follows that $g_{23}g_{32}\ne0_\F$.
[If $g_{23}g_{32}=0_\F$, then $g_{22}g_{33}=0_\F$ since $\beta-\xi\ne0_\F$ (we have assumed that $\beta\not\in\overline{\xi}$), leading to $\gamma=0_\F$.]
But $1_\F-\beta\xi\ne0_\F$ (again by the assumption that $\beta\not\in\overline{\xi}$), so $(1_\F-\beta\xi)g_{23}g_{32}\ne0_\F$.
Hence $(\beta-\xi)g_{22}g_{33}\ne0_\F$.
We conclude that $g_{22}$ and $g_{32}$ are both non-zero, leading to $\mu_{221}\ne0_\F$ (since $\gamma\ne0_\F$ and $1_\F+\xi\ne0_\F$ by assumption), which is a contradiction.
\end{proof}

\begin{lemma}\label{Lemma3.2}
Let $\xi\in\F-\{-1_\F\}$ and let $\boldsymbol\rho={\bf221}+(2\cdot1_\F){\bf321}+{\bf331}\in\boldsymbol\Lambda$.
Then $\boldsymbol\sigma(\xi)\,(={\bf231}+\xi{\bf321})\not\in \overline{O(\boldsymbol\rho)}$.
\end{lemma}


\begin{proof}
Assume the hypothesis.
Beginning with the commutation relations (relative to the standard basis $(e_1,e_2,e_3)$ of $V$) for $\Theta^{-1}(\boldsymbol\rho)$ we see, via the change of basis $e_1'=e_1$, $e_2'=e_2-e_3$, $e_3'=e_3$, that $\boldsymbol\nu=-{\bf231}+{\bf321}+{\bf331}\in O(\boldsymbol\rho)$.
Moreover, for $b\,(=(b_{ij}))\in B$ we get that
\[
\boldsymbol\nu b=-\Big(\frac{b_{22}b_{33}}{b_{11}}\Big){\bf231}+\Big(\frac{b_{22}b_{33}}{b_{11}}\Big){\bf321}+\Big(\frac{b_{33}^2}{b_{11}}\Big){\bf331}, \mbox{\ with\ } b_{11}b_{22}b_{33}\ne0_\F.
\]

It follows that the polynomials $X_{231}+X_{321}$ and $X_{ijk}$ for $(i,j,k)\in \mathcal T-\{(2,3,1),\, (3,2,1),\, (3,3,1)\}$ of $\F[{\bf X}]$ all belong to the vanishing ideal ${\bf I}(\boldsymbol\nu B)$.

We suppose, on the contrary, that $\boldsymbol\sigma(\xi)\in\overline{O(\boldsymbol\nu)}\,(=\overline{O(\boldsymbol\rho)})$.
Then, in view of Result~\ref{Result2.3}, we must have that $\boldsymbol\sigma(\xi) g\in\overline{\boldsymbol\nu B}$ for some $g\,(=(g_{ij}))\in G$.
Set $\boldsymbol\mu\,(=(\mu_{ijk}))=\boldsymbol\sigma(\xi)g\in\boldsymbol\Lambda$ and  let $\gamma=(\det g)^{-1}(g_{22}g_{33}-g_{23}g_{32})\in\F$.
Comparing with the proof of Lemma~\ref{Lemma_n4_are_non-isomorphic} we get, in view of the observations above, that
\begin{gather*}
\mu_{231}=\gamma(g_{22}g_{33}+\xi g_{23}g_{32}),\qquad \mu_{321}=\gamma(g_{23}g_{32}+\xi g_{22}g_{33})=-\mu_{231},\\
\mu_{331}=\gamma (1_\F+\xi)g_{23}g_{33},\qquad\mu_{221}=\gamma(1_\F+\xi) g_{22}g_{32}=0_\F,
\end{gather*}
with all other $\mu_{ijk}$ being equal to zero.
Moreover, we can assume that $\gamma\ne0_\F$ since the assumption $\gamma=0_\F$ leads to $\boldsymbol\mu=\bf0$ which is a contradiction.
From $\mu_{321}+\mu_{231}=0_\F$ and $\gamma\ne0_\F$ we get that $(1_\F+\xi)(g_{22}g_{33}+g_{23}g_{32})=0_\F$.
Moreover, the assumption $1_\F+\xi\ne0_\F$ leads to $g_{22}g_{33}+g_{23}g_{32}=0_\F$.
Now let $\gamma'=(\det g)^{-1}(g_{22}g_{33}+g_{23}g_{32})$.
Then $\gamma'=0_{\F}$.
If ${\rm char}\,\F=2$, we have $\gamma=\gamma'$ with $\gamma\ne0_\F$ and $\gamma'=0_\F$, a contradiction.
In the case ${\rm char}\,\F\ne2$ we get, by looking at $\gamma'+\gamma$ (which is non-zero), that $g_{22}g_{33}\ne0_\F$.
Invoking the fact that $\gamma'=0_\F$ we also get that $g_{23}g_{32}\ne0_\F$.
In particular, $g_{22}\ne0_\F$ and $g_{32}\ne0_\F$ giving $\mu_{221}\ne0_\F$, since $(1_\F+\xi)\gamma\ne0_\F$.
We have again arrived at a contradiction.
This establishes that $O(\boldsymbol\sigma(\xi))\cap\overline{\boldsymbol\nu B}=\varnothing$.
We conclude that $\boldsymbol\sigma(\xi)\not\in\overline{O(\boldsymbol\nu)}\,(=\overline{O(\boldsymbol\rho)})$ whenever $\xi\ne-1_\F$.
\end{proof}

\section{3-dimensional nilpotent associative  algebras}\label{Section3dimAlg}

Recall the running assumption that $\F$ is algebraically closed.
Also recall that $(e_1,e_2,e_3)$ is the standard basis of the underlying 3-dimensional $\F$-vector space $V$ we have fixed in Section~\ref{sec:Prilim}.
With this assumption on $\F$ we immediately get, based on~\cite[Theorem~2.3.6]{KrusePrice1969} or~\cite[Section~5]{deGraaf2018}, the classification of 3-dimensional nilpotent associative algebras given in Table~\ref{Table3dimAssocClass}.

\setcounter{tbn}{0}

{\begin{center}
\footnotesize
\refstepcounter{tbn}\label{Table3dimAssocClass}

\begin{longtable}{|c|l| }
\hline
$\mathfrak a\in\Theta^{-1}(\mathcal N)$   & Non-zero commutation relations relative to the standard basis of $V$  \\
\hline
 $\mathfrak a_0$&  Abelian (zero algebra)  \\
\hline
 $\mathfrak c_1$   & $e_3e_3=e_1$   \\
\hline
 $\mathfrak c_3$   & $e_2e_2=e_1$, $e_3e_3= e_1$  \\
\hline
 $\mathfrak a(\delta)$, $\delta\in\F$  & $e_2e_2=e_1$, $e_3e_3=\delta e_1$, $e_2e_3=e_1$    \\
\hline
 $\mathfrak l_1$  & $e_2e_3=e_1$, $e_3e_2=-e_1$   \\
\hline
 $\mathfrak c_5$  & $e_1e_1=e_2$, $e_1e_2=e_3$, $e_2e_1=e_3$   \\
\hline
\end{longtable}

Table~\ref{Table3dimAssocClass}:
Non-isomorphic 3-dimensional nilpotent associative algebras over an algebraically closed field $\F$ (the information taken from~\cite{KrusePrice1969} or~\cite{deGraaf2018}).
We write $e_ie_j$ in place of $[e_i,e_j]_{\mathfrak a}$ \ for $\mathfrak a\in\Theta^{-1}(\mathcal N)$.

\end{center}}

\begin{remark}\label{rem:commentsOnTable1}
(Some comments on Table~\ref{Table3dimAssocClass}.)

(i) 
Recalling the definition of the set $\mathcal N$ in Section~\ref{sec:Prilim}, we have that as $\mathfrak a$ runs through the elements of the first column of Table~\ref{Table3dimAssocClass}, the structure vector $\Theta(\mathfrak a)$ runs through a complete system of representatives for the distinct $G$-orbits inside $\mathcal N$.
(Clearly the set $\mathcal N$ is a union of $G$-orbits.)

\smallskip
(ii) About $\mathfrak c_3$: Since $\F$ is algebraically closed we get that $(\F^*)^2=\F^*$, so in item (4) of~\cite[Theorem~2.3.6]{KrusePrice1969} there is exactly one isomorphism class of algebras, given for example by choosing $\gamma=1_\F$.
Moreover, by setting $\hat{\mathfrak c}_3=\Theta^{-1}({\bf231}+{\bf321})\in\boldsymbol A$ we get, in the case ${\rm char}\,\F\ne2$, that $\hat{\mathfrak c}_3\simeq \mathfrak c_3$.
Beginning with the commutation relations for $\mathfrak c_3$ relative to the standard basis $(e_1,e_2,e_3)$ of $V$ (see Table~\ref{Table3dimAssocClass}), in order to go from $\Theta(\mathfrak c_3)$ to $\Theta(\hat{\mathfrak c}_3)$ consider the change of basis $e_1'=2e_1$, $e_2'=e_2-\omega e_3$, $e_3'=e_2+\omega e_3$, where $\omega\in\F$ is a root of polynomial $x^2+1\in\F[x]$.

\smallskip
(iii) It will be convenient, in order to be able to compare the results obtained in this paper directly with the results in~\cite{IvanovaPallikaros2022}, to define in $\boldsymbol A$ algebras $\mathfrak a_2$ and $\mathfrak a_3(\kappa)$, for each $\kappa\in\mathbb F$, by $\mathfrak a_2=\Theta^{-1}({\bf231})$ and $\mathfrak a_3(\kappa)=\Theta^{-1}({\bf221}+\kappa{\bf321}+{\bf331})$.
We can then see immediately that $\mathfrak c_3=\mathfrak a_3(\kappa=0_\F)$.
Moreover, starting from the defining condition for $\mathfrak a_2$ and considering the change of basis $e_1'=e_1$, $e_2'=e_2+e_3$, $e_3'=e_3$ we get that $\mathfrak a_2\simeq \mathfrak a(\delta=0_\F)$.
Finally note that for $\kappa\in\F-\{0_\F\}$, now beginning with the commutation relations of $\mathfrak a_3(\kappa)$ relative to the standard basis of $V$, the change of basis $e_1'=e_1$, $e_2'=-e_2$, $e_3'=e_3$ gives $\mathfrak a_3(\kappa)\simeq\mathfrak a_3(-\kappa)$.
In fact we have equality in this last expression in the case ${\rm char}\,\F=2$.

\medskip
(iv) Comparing with the discussion in Section~\ref{SectionTwoLemmas}, for each $\beta\in\F$ we define $\mathfrak h(\beta)\in\boldsymbol A$ as the algebra structure satisfying $\mathfrak h(\beta)=\Theta^{-1}({\bf231}+\beta{\bf321})=\Theta^{-1}(\boldsymbol\sigma(\beta))$.
Then $\mathfrak h(\beta=1_\F)=\hat{\mathfrak c}_3$, $\mathfrak h(\beta=0_\F)=\mathfrak a_2\,(\simeq\mathfrak a(\delta=0_\F)$ from item~(iii)), and $\mathfrak h(\beta=-1_\F)=\mathfrak l_1$.
Note that $\hat{\mathfrak c}_3=\mathfrak l_1$ in the special case ${\rm char}\,\F=2$.
\end{remark}

In the sequence of lemmas that follow (and their corollaries), we include some further observations regarding the algebras in Table~\ref{Table3dimAssocClass}.

\begin{lemma}\label{lemma2.11}
Let $\beta,\beta'\in\F$.
We have that $\mathfrak h(\beta)\simeq\mathfrak h(\beta')$ if, and only if, $\beta'\in\overline{\beta}$.
\end{lemma}
\begin{proof}
We first assume that $\beta\beta'=1_\F$.
In particular, $\beta,\beta'\in\F-\{0_{\F}\}$.
The non-zero commutation relations (relative to the standard basis $(e_1,e_2,e_3)$ of $V$) in $\mathfrak h(\beta)$ are: $e_2e_3=e_1$, $e_3e_2=\beta e_1$.
Consider the change of basis given by $e_1'=\beta e_1$, $e_2'=e_3$, $e_3'=e_2$.
Then the non-zero commutation relations of $\mathfrak h(\beta)$ relative to the ordered basis $(e_1', e_2', e_3')$ of $V$ are given by
\[
e_2'e_3'=e_3e_2=\beta e_1=e_1',\quad
e_3'e_2'=e_2e_3=e_1=(\beta^{-1})e_1'=\beta'e_1'.
\]
It follows that the structure vector of $\mathfrak h(\beta)$ relative to the basis $(e_1', e_2', e_3')$ is ${\bf231}+\beta'{\bf321}$ which is equal to $\Theta(\mathfrak h(\beta'))$.
We conclude that $\mathfrak h(\beta)\simeq\mathfrak h(\beta')$ whenever  $\beta'\in\overline{\beta}$ (the result is obvious if $\beta=\beta'$).

The converse is immediate from Lemma~\ref{Lemma_n4_are_non-isomorphic}.
\end{proof}



At this point it would be useful to recall the notation $m\cdot1_{\F}$ (for $m\in \mathbb Z-\{0\}$) introduced in Section~\ref{sec:Prilim}.
In what follows, for $m>0$, we will write $-m\cdot1_{\F}$ in the place of $(-m)\cdot1_{\F}\,(=-(m\cdot1_{\F}))$.
In particular, in the case ${\rm char}\,\F\ne2$ the element $2\cdot1_{\F}$ of $\F$ is non-zero and $-2\cdot1_{\F}\ne 2\cdot1_{\F}$.

\begin{lemma}\label{Lemma4.3}
Assume that ${\rm char}\,\F\ne2$.
If $\beta\in\F-\{0_\F,1_\F,-1_\F\}$, the $\mathfrak h(\beta)\simeq\mathfrak a_3(\kappa)$ for some $\kappa\in\F-\{0_\F,2\cdot1_\F,-2\cdot1_\F\}$.
Conversely, if $\kappa\in\F-\{0_\F,2\cdot1_\F,-2\cdot1_\F\}$, then $\mathfrak a_3(\kappa)\simeq\mathfrak h(\beta)$ for some $\beta\in\F-\{0_\F,1_\F,-1_\F\}$.
\end{lemma}
\begin{proof}
It will be convenient to define, for each $\alpha\in\F-\{0_{\F}\}$, the matrix $g_\alpha\in M_3(\F)$ by
$g_\alpha=${\footnotesize$\begin{pmatrix}\alpha-\alpha^{-1}&0_\F&0_\F\\0_\F&\alpha&1_\F\\0&1_\F&\alpha\end{pmatrix}$}.
Then $g_\alpha\in\GlF$ precisely when $\alpha\in\F-\{0_\F,1_\F,-1_\F\}$.

Suppose first that $\beta\in\F-\{0_\F,1_\F,-1_\F\}$.
Then, since $\F$ is algebraically closed, $\beta=-\alpha^2$ for some $\alpha\in\F-\{0_\F,1_\F,-1_\F,\omega,-\omega\}$ where $\omega^2+1_\F=0_\F$.
[Clearly, we have that $\alpha\in\{0_\F,1_\F,-1_\F,\omega,-\omega\}$ if, and only if, $-\alpha^2\in\{0_\F,1_\F,-1_\F\}$.]
Moreover, we observe that $({\bf231}+\beta{\bf321})g^{-1}_\alpha=({\bf231}+-\alpha^2{\bf321})g^{-1}_\alpha={\bf221}+{\bf331}+-(\alpha+\alpha^{-1}){\bf321}=\Theta(\mathfrak a_3(\kappa))$ where $\kappa=-(\alpha+\alpha^{-1})\in\F$ $-$ $\{0_\F$, $2\cdot1_\F,-2\cdot1_\F\}$.
[Note that (under the assumption $\alpha\ne0_\F$), we have $\kappa=0_\F$ iff $\alpha^2+1_\F=0_\F$ iff $\alpha\in\{\omega,-\omega\}$, while $\kappa=2\cdot1_\F$ iff $(\alpha+1_\F)^2=0_\F$ iff $\alpha=-1_\F$ and, finally $\kappa=-2\cdot1_\F$ iff $(\alpha-1_\F)^2=0_\F$ iff $\alpha=1_\F$.]

Conversely, we suppose now that $\kappa\in\F-\{0_\F,2\cdot1_\F,-2\cdot1_\F\}$.
Let $\alpha\in\F$ be a root of the polynomial $x^2+\kappa x+1_\F\in\F[x]$.
Clearly $\alpha\ne0_\F$ and, from the condition $\alpha^2+\kappa\alpha+1_\F=0_\F$, we get that $\kappa=-(\alpha+\alpha^{-1})$.
From the discussion in the first part of the proof and the assumption that $\kappa\not\in\{0_\F,2\cdot1_\F,-2\cdot1_\F\}$, we can observe that $\alpha\not\in\{0_\F,1_\F,-1_\F,\omega,-\omega\}$.
Moreover, we see that $({\bf231}+{\bf331}+\kappa{\bf321})g_\alpha=({\bf221}+{\bf331}+-(\alpha+\alpha^{-1}){\bf321})g_\alpha={\bf231}+-\alpha^2{\bf321}$.
Hence, $\mathfrak a_3(\kappa)\simeq\mathfrak h(-\alpha^2)$ for some $\alpha\in\F-\{0_\F,1_\F,-1_\F,\omega,-\omega\}$.
By setting $\beta=-\alpha^2$, we conclude that $\mathfrak a_3(\kappa)\simeq\mathfrak h(\beta)$ for some $\beta\in\F-\{0_\F,1_\F,-1_\F\}$
(compare again with the first part of the proof).
\end{proof}

%
%

Now let $\delta\in\F-\{0_\F\}$, making no assumption on the characteristic of $\F$.
Beginning with the defining conditions for $\mathfrak a(\delta)$ relative to the standard basis of $V$, and considering the change of basis $e_1'=e_1$, $e_2'=\kappa e_3$, $e_3'=e_2$ where $\kappa\in\F-\{0_\F\}$ is a root of the polynomial $x^2-\delta^{-1}\in\F[x]$, we get that $\mathfrak a(\delta)\simeq \mathfrak a_3(\kappa)$ with $\kappa^2=\delta^{-1}$.
(Recall our assumption that $\F$ is algebraically closed.)

In other words, for $\delta\in\F-\{0_\F\}$ and $g_\kappa=${\footnotesize$\begin{pmatrix}1_\F&0_\F&0_\F\\0_\F&0_\F&1_\F\\0_\F&\kappa&0_\F\end{pmatrix}$}$\in\GlF$ (where $\kappa\in\F-\{0_\F\}$ satisfies $\kappa^2=\delta^{-1}$) we have $({\bf221}+{\bf231}+\delta{\bf331})g_\kappa={\bf221}+\kappa{\bf321}+{\bf331}$.

Suppose now that $\kappa\in\F-\{0_\F\}$.
By setting $\delta=1_\F/\kappa^2\,(\in\F-\{0_\F\})$, we see that $({\bf221}+\kappa{\bf321}+{\bf331})g_\kappa^{-1}={\bf221}+{\bf231}+\delta{\bf331}$.
We have proved the following:

\begin{lemma}\label{lemma_aDelta_sim_a3K}
If $\delta\in\F-\{0_\F\}$, then $\mathfrak a(\delta)\simeq \mathfrak a_3(\kappa)$ for some $\kappa\in\F-\{0_\F\}$.
Conversely, if $\kappa\in\F-\{0_\F\}$, then $\mathfrak a_3(\kappa)\simeq \mathfrak a(\delta)$ for some $\delta\in\F-\{0_\F\}$.
In each of these cases, $\kappa$ and $\delta$ are related by $\kappa^2=\delta^{-1}$.
\end{lemma}

\begin{remark}\label{Remark4.5}
Assume that ${\rm char}\,\F\ne2$.
We can make the following observations.

(i) Since distinct $\delta\in\F$ correspond to distinct (up to isomorphism) algebras $\mathfrak a(\delta)$, we see, for $\kappa_1,\kappa_2\in\F-\{0_\F\}$, that $\mathfrak a_3(\kappa_1)\simeq\mathfrak a_3(\kappa_2)$ if, and only if, $\kappa_2\in\{\kappa_1,-\kappa_1\}$.
[Since $\mathfrak a_3(\kappa)\simeq \mathfrak a(1_\F/\kappa^2)$ for each $\kappa\in\F-\{0_\F\}$, we get, (for $\kappa_1,\kappa_2\ne0_\F$), that $\mathfrak a_3(\kappa_1)\simeq\mathfrak a_3(\kappa_2)$ if, and only if, $1_\F/\kappa_1^2=1_\F/\kappa_2^2$ \
and this last holds if, and only if, $\kappa_1^2-\kappa_2^2=0_\F$.]
Compare also with Remark~\ref{rem:commentsOnTable1}(iii).

(ii) Let $\delta\in\F-\{0_\F\}$.
It is immediate from item (i) of this remark and Lemma~\ref{lemma_aDelta_sim_a3K} that $\mathfrak a(\delta)\simeq \mathfrak a_3(2\cdot 1_\F)\,(\simeq \mathfrak a_3(-2\cdot 1_\F))$ if, and only if, $\delta=(4\cdot1_\F)^{-1}$.
[Note that $2\cdot1_\F\ne0_\F$ and $4\cdot1_\F\ne0_\F$ since ${\rm char}\,\F\ne2$.
Moreover, we have that $(2\cdot1_\F)(2\cdot1_\F)=4\cdot1_\F=(-2\cdot1_\F)(-2\cdot1_\F)$.]
\end{remark}

\begin{corollary}\label{corol:aDelta_sim_a3K}
Assume that ${\rm char}\,\F\ne2$.
Let $\delta\in\F-\{0_\F, (4\cdot1_\F)^{-1}\}$.
Then $\mathfrak a(\delta)\simeq \mathfrak a_3(\kappa)$ for some $\kappa\in\F-\{0_\F, 2\cdot1_\F,-2\cdot1_\F\}$.
Conversely, if $\kappa\in\F-\{0_\F, 2\cdot1_\F,-2\cdot1_\F\}$, then $\mathfrak a_3(\kappa)\simeq \mathfrak a(\delta)$ for some $\delta\in\F-\{0_\F, (4\cdot1_\F)^{-1}\}$.
\end{corollary}

Recalling that $\mathfrak h(\beta=0_\F)=\mathfrak a_2\simeq\mathfrak a(\delta=0_\F)$ we get, in view of Lemmas~\ref{lemma2.11}, \ref{Lemma4.3} and Corollary~\ref{corol:aDelta_sim_a3K}, the following:

\begin{corollary}\label{Corol4.7}
Assume that ${\rm char}\,\F\ne2$.
If $\delta\in\F-\{(4\cdot1_\F)^{-1}\}$,
then $\mathfrak a(\delta)\simeq \mathfrak h(\beta)$ for some $\beta\in\F-\{1_\F,-1_\F\}$.
Conversely, if $\beta\in\F-\{1_\F,-1_\F\}$, then $\mathfrak h(\beta)\simeq\mathfrak a(\delta)$ for some $\delta\in\F-\{(4\cdot1_\F)^{-1}\}$.
\end{corollary}

For the rest of the discussion in this section we assume that ${\rm char}\,\F=2$.
Then $\gamma=-\gamma$ (equivalently, $2\cdot \gamma=0_\F$) for each $\gamma\in\F$.
Moreover, for $\omega\in\F$, we have that $\omega^2+1_\F=0_\F$ precisely when $\omega=1_\F$ under this assumption.
Thus, going through the proof of Lemma~\ref{Lemma4.3}, but now assuming that ${\rm char}\,\F=2$ instead and making the appropriate changes, we obtain the following result.

\begin{lemma}\label{Lemma4.8char2Lemma4.3}
Assume that ${\rm char}\,\F=2$.
If $\beta\in\F-\{0_\F,1_\F\}$, then $\mathfrak h(\beta)\simeq\mathfrak a_3(\kappa)$ for some $\kappa\in\F-\{0_\F\}$.
Conversely, if $\kappa\in\F-\{0_\F\}$, then $\mathfrak a_3(\kappa)\simeq\mathfrak h(\beta)$ for some $\beta\in\F-\{0_\F,1_\F\}$.
\end{lemma}


\begin{corollary}\label{Corol4.9}
Assume that ${\rm char}\,\F=2$.
If $\delta\in\F$, we have that $\mathfrak a(\delta)\simeq\mathfrak h(\beta)$ for some $\beta\in\F-\{1_\F\}$.
Conversely, if $\beta\in\F-\{1_\F\}$, then $h(\beta)\simeq\mathfrak a(\delta)$ for some $\delta\in\F$.
\end{corollary}
\begin{proof}
First, recall from Remark~\ref{rem:commentsOnTable1}(iii),(iv) that $\mathfrak h(\beta=0_\F)=\mathfrak a_2\,(\simeq\mathfrak a(\delta=0_\F))$.
Now let $\delta\in\F-\{0_\F\}$.
From Lemma~\ref{lemma_aDelta_sim_a3K} (for which we have no restriction on ${\rm char}\,\F$), we get $\mathfrak a(\delta)\simeq\mathfrak a_3(\kappa)$ for some $\kappa\in\F-\{0_\F\}$, and under this last assumption for $\kappa$, we have that $\mathfrak a_3(\kappa)\simeq\mathfrak h(\beta)$ for some $\beta\in\F-\{0_\F,1_\F\}$ in view of Lemma~\ref{Lemma4.8char2Lemma4.3}.
Next we assume that $\beta\in\F-\{0_\F, 1_\F\}$.
Again from Lemma~\ref{Lemma4.8char2Lemma4.3}  we get that $h(\beta)\simeq\mathfrak a_3(\kappa)$ for some $\kappa\in\F-\{0_\F\}$.
Hence, $\mathfrak h(\beta)\simeq\mathfrak a(\delta)$ for some $\delta\in\F-\{0_\F\}$ by Lemma~\ref{lemma_aDelta_sim_a3K}.
The required result now follows easily by combining the three observations above.
\end{proof}

\section{The degeneration picture inside the set $\mathcal N$}

We first observe that the defining conditions for an algebra $\mathfrak g\in\boldsymbol A$ to be nilpotent or associative are equivalent to certain systems of polynomial equations in the coefficients of $\Theta(\mathfrak g)\in\boldsymbol\Lambda$.
In particular, if $\Theta(\mathfrak g)=(\lambda_{ijk})$, the condition for associativity is equivalent to $\sum_{i=1}^3\lambda_{bci}\lambda_{aid}=\sum_{i=1}^3\lambda_{abi}\lambda_{icd}$ for $a,b,c,d\in\{1,2,3\}$, while, according to the definitions we have given in Section~\ref{sec:Prilim}, the condition for $\mathfrak g\in\boldsymbol A$ to be nilpotent of class at most 2 (resp., nilpotent of class at most 3) is equivalent to $\sum_{i=1}^3\lambda_{abi}\lambda_{icd}=0_\F$ for $a,b,c,d\in\{1,2,3\}$ (resp., $\sum_{j,k=1}^3\lambda_{abk}\lambda_{kcj}\lambda_{jde}=0_\F$ for $a,b,c,d,e\in\{1,2,3\}$).
It follows that the set $\mathcal N$ is an algebraic subset of~$\boldsymbol\Lambda$ which is clearly also $G$-invariant relative to the action of $G$ on $\boldsymbol\Lambda$ we are considering.
Similarly for the subsets of $\mathcal N$ corresponding to algebras which are nilpotent of class at most $r$, for each positive integer $r$.

\smallskip
It will be convenient at this point to define the subsets $\mathcal C$ and $\mathcal M^{**}$ of $\boldsymbol\Lambda$ by
\begin{gather*}
\Theta^{-1}(\mathcal C)=\{\mathfrak g\in\boldsymbol A\colon [u,v]_{\mathfrak g}=[v,u]_{\mathfrak g}\ \mbox{for all\ } u,v\in V\},\\
\Theta^{-1}(\mathcal M^{**})=\{\mathfrak g\in\boldsymbol A\colon [u,u]_{\mathfrak g}\in\F\mbox{-span}(u)\ \mbox{for all\ } u\in V\}.
\end{gather*}
Then both $\mathcal C$ and $\mathcal M^{**}$ are $G$-invariant Zariski-closed subsets of $\boldsymbol\Lambda$,
(see~\cite[Sections~3.1 and~6.1]{PallikarosWard2020}).

\begin{result}\label{result:2.18}(See~\cite[Lemmas~4.10, 5.3 and~5.4]{IvanovaPallikaros2019}, proved there over an arbitrary field.)


(i) $\overline{O(\Theta(\mathfrak l_1))}=O(\Theta(\mathfrak l_1))\cup\{{\bf0}\}$ \ and \ $\overline{O(\Theta(\mathfrak c_1))}=O(\Theta(\mathfrak c_1))\cup\{{\bf0}\}$

(ii) If $\boldsymbol\lambda\in\boldsymbol\Lambda-\mathcal M^{**}$, then $\Theta(\mathfrak c_1)\in\overline{O(\boldsymbol\lambda)}$.
\end{result}

\begin{remark}\label{Remark:5.2New}
(i) From Result~\ref{result:2.18}(ii) we get that $\mathfrak c_3\to\mathfrak c_1$ and $\mathfrak a(\delta)\to\mathfrak c_1$ for all $\delta\in\F$, since $\{\mathfrak c_3\}\cup \{\mathfrak a(\delta)\colon \delta\in\F\}\subseteq \boldsymbol A-\Theta^{-1}(\mathcal M^{**})$.

(ii) By examining nilpotency classes we can observe that none of the other algebras under consideration degenerates to $\mathfrak c_5$.
In particular, apart from algebra $\mathfrak c_5$ which is nilpotent of class 3, all the remaining ones are nilpotent of class at most 2.
\end{remark}

In obtaining the degeneration picture inside the set $\mathcal N$ there is a natural split between the cases ${\rm char}\,\F\ne2$ and ${\rm char}\,\F=2$.
These two cases, which as it turns out give rise to different results, are considered separately in Subsections~\ref{subsection5.1} and~\ref{subsection5.2} that follow.
Recall our running assumption that the field $\F$ is algebraically closed in the discussion of both these subsections.

\subsection{The case ${\rm char}\,\F\ne2$}\label{subsection5.1}

{\it For the whole of this subsection assume that ${\rm char}\,\F\ne2$.}

With this assumption on the characteristic of $\F$, the degeneration picture inside $\mathcal N$ can now be completed using Result~\ref{Result2.1} (special case $H=G$) together with Results~\ref{Result2.2} and~\ref{result:2.18}, Remark~\ref{Remark:5.2New} and the observations (i)--(iv) below.

\smallskip
(i) We can observe that $\mathcal C\cap\mathcal N=\{{\bf0}\}\cup O(\Theta(\mathfrak c_1))\cup O(\Theta(\mathfrak c_3))\cup O(\Theta(\mathfrak c_5))$.
In particular this set is algebraic as it is the intersection of two algebraic sets.

\smallskip
(ii) Set $\boldsymbol\lambda=\Theta(\mathfrak c_5)\,(={\bf112}+{\bf123}+{\bf213})\in\boldsymbol\Lambda$.
For each $t\in\F-\{0_\F\}$, define $g(t)=${\footnotesize$\begin{pmatrix}t&0_\F&0_\F\\t&1_\F&0_\F\\0_\F&0_\F&t\end{pmatrix}$}$\in\GlF$.
We can then easily compute $\boldsymbol\lambda g(t)=t^2{\bf112}+(2\cdot1_\F)\,t\,{\bf113}+{\bf123}+{\bf213}$, for $t\in\F-\{0_\F\}$.
The map $f\colon\F\to\boldsymbol\Lambda:$ $t\mapsto t^2{\bf112}+(2\cdot1_\F)\,t\,{\bf113}+{\bf123}+{\bf213}$ is continuous in the Zariski topology.
Set $S=\F-\{0_\F\}$.
Clearly $f(S)\subseteq O(\boldsymbol\lambda)$.
Invoking Result~\ref{Result1} we get that ${\bf123}+{\bf213}\,(=f(0_\F))\in\overline{f(S)}\subseteq\overline{O(\boldsymbol\lambda)}$.
We conclude that $\mathfrak c_5\to\mathfrak c_3$ (see Remark~\ref{rem:commentsOnTable1}(ii)).
In a similar manner, in~\cite[Example~2.3]{IvanovaPallikaros2022} (working in that paper over an arbitrary field of characteristic $\ne2$ for this particular example) it was shown that $\mathfrak a_3(\kappa=2\cdot1_\F)\to \mathfrak l_1$.
In view of Remark~\ref{Remark4.5} we get that $\mathfrak a(\delta=(4\cdot1_\F)^{-1})\to\mathfrak l_1$.

\smallskip
(iii) Combining Lemmas~\ref{Lemma_n4_are_non-isomorphic} and~\ref{lemma2.11} with Remark~\ref{rem:commentsOnTable1}(iv) and Corollary~\ref{Corol4.7}, we get that there is no degeneration between any two distinct members of the set $\{\mathfrak l_1\}\cup\{\mathfrak c_3\}\cup\{\mathfrak a(\delta)\colon \delta\in\F-\{(4\cdot1_\F)^{-1}\}\}$.
It now follows from Result~\ref{Result2.2}(i) that $\mathfrak a(\delta\ne(4\cdot1_\F)^{-1})\not\to\mathfrak c_5$ (as we have already observed considering nilpotency classes) and that $\mathfrak a(\delta\ne(4\cdot1_\F)^{-1})\not\to\mathfrak a(\delta=(4\cdot1_\F)^{-1})$.

\smallskip
(iv) In Lemma~\ref{Lemma3.2} (see also Remark~\ref{Remark4.5}(ii)) we have shown that $\mathfrak a(\delta=(4\cdot1_\F)^{-1})\not\to\Theta^{-1}({\bf231}+\xi{\bf321})$ whenever $\xi\in\F-\{-1_\F\}$.
Invoking Corollary~\ref{Corol4.7} we see that that there is no degeneration from $\mathfrak a(\delta=(4\cdot1_\F)^{-1})$ to any one of the members of the set $\{\mathfrak c_3\}\cup\{\mathfrak a(\delta)\colon \delta\in\F-\{(4\cdot1_\F)^{-1}\}\}$.
Finally, note that by using Result~\ref{Result2.2}(i)  we can also verify that $\mathfrak a(\delta=(4\cdot1_\F)^{-1})\not\to\mathfrak c_5$, a fact already obtained above.

\bigskip

\begin{center}
\begin{tikzpicture}
\matrix [column sep=7mm, row sep=5mm] {
\node (n5)  {$\mathfrak c_{5}$}; & &&& \\
\node (n3)  {$\mathfrak c_{3}$}; & &  \node (n4muNot1/4)  {$\mathfrak a(\delta\ne(4\cdot1_\F)^{-1})$}; & &  \node (n4mu1/4)  {$\mathfrak a(\delta=(4\cdot1_\F)^{-1})$};
 \\
& \node (n2)  {$\mathfrak c_1$}; & & \node (n1)  {$\mathfrak l_{1}$};  &
   \\
 & & \node (n0)   {$\mathfrak a_{0}$}; &  &
 \\
};
\draw[->, thin] (n5) -- (n3);
\draw[->, thin] (n3) -- (n2);
\draw[->, thin] (n4muNot1/4) -- (n2);
\draw[->, thin] (n4mu1/4) -- (n2);
\draw[->, thin] (n4mu1/4) -- (n1);
\draw[->, thin] (n1) -- (n0);
\draw[->, thin] (n2) -- (n0);
\end{tikzpicture}

{\footnotesize
\refstepcounter{pict}\label{Picture3DimDegenAssocCharNe2}
Picture~\ref{Picture3DimDegenAssocCharNe2}.
Degenerations of 3-dimensional nilpotent associative algebras over an algebraically closed field $\F$ of ${\rm char}\,\F\ne2$ (with $\delta\in\F$)
}
\end{center}

We thus obtain the above degeneration picture inside the set $\mathcal N$ (defined over an algebraically closed field of characteristic not equal to 2).
In the special case $\F=\mathbb C$ these results corroborate, by making use of different techniques for some of the arguments, the corresponding results in~\cite{IvanovaPallikaros2022}.

\subsection{The case ${\rm char}\,\F=2$}\label{subsection5.2}

{\it For the whole of this subsection assume that ${\rm char}\,\F=2$.}

For the discussion that follows it will be convenient to define the elements $\boldsymbol\lambda$,  $\boldsymbol\mu$,  $\boldsymbol\nu\in \boldsymbol\Lambda$ by $\boldsymbol\lambda={\bf231}+{\bf321}+{\bf332}$, $\boldsymbol\mu={\bf231}+{\bf321}+{\bf331}$ and $\boldsymbol\nu={\bf231}+{\bf321}$.
Note that $\Theta(\mathfrak c_5)\in O(\boldsymbol\lambda)$, $\Theta(\mathfrak c_3)\in O(\boldsymbol\mu)$ and $\boldsymbol\nu=\Theta(\mathfrak l_1)$ under the assumption we have on the characteristic.
(To see this, compare with the relations given in Table~\ref{Table3dimAssocClass}: Beginning with $\Theta(\mathfrak c_5)$, the structure vector of $\mathfrak c_5$ relative to the standard basis $(e_1,e_2,e_3)$ of $V$, consider the change of basis $e_1'=e_3$, $e_2'=e_2$ and $e_3'=e_1$.
Similarly, beginning with $\Theta(\mathfrak c_3)$ consider the change of basis $e_1''=e_1$, $e_2''=e_2+e_3$ and $e_3''=e_3$.)

We can thus now make the following observations --- note in particular the various differences with the ${\rm char}\,\F\ne2$ case considered above.

\medskip
(i) As before, the subset $\mathcal C\cap\mathcal N$ of $\mathcal N$ is algebraic.
However, under our present hypothesis on the characteristic of $\F$, we now have that $\mathcal C\cap\mathcal N=\{{\bf 0}\}\cup O(\Theta(\mathfrak c_1))\cup O(\Theta(\mathfrak l_1))\cup O(\Theta(\mathfrak c_3))\cup O(\Theta(\mathfrak c_5))=\mathcal N-\bigcup_{\delta\in\F}O(\Theta(\mathfrak a(\delta)))$.

\medskip
(ii) For each $t\in\F-\{0_\F\}$, define matrix $g(t)\in\GlF$ by $g(t)=${\footnotesize$\begin{pmatrix}t^2&t&0\\0&t&0\\0&0&t\end{pmatrix}$}.
Then $\boldsymbol\lambda g(t)={\bf231}+{\bf321}+{\bf331}+t\,{\bf332}$, with $t\in\F-\{0_\F\}$.
By similar argument as in Observation~(ii) of Subsection~\ref{subsection5.1}, we get that $\boldsymbol\mu\in\overline{O(\boldsymbol\lambda)}$.
We conclude that $\mathfrak c_5\to\mathfrak c_3$.

\medskip
(iii) Similarly, defining matrix $g(t)\in\GlF$ for each $t\in\F-\{0_\F\}$ by $g(t)=${\footnotesize$\begin{pmatrix}t&0&0\\0&1&0\\0&0&t\end{pmatrix}$}, we get that $\boldsymbol\mu g(t)={\bf231}+{\bf321}+t\,{\bf331}=\boldsymbol\nu+t\,{\bf331}$, which in turn gives that $\mathfrak c_3\to\mathfrak l_1$.

\medskip
(iv) Combining Lemmas~\ref{Lemma_n4_are_non-isomorphic} and~\ref{lemma2.11} with Corollary~\ref{Corol4.9}, we get that there is no degeneration between any two distinct members of the family $\mathfrak \{\mathfrak l_1\}\cup\{\mathfrak a(\delta)\colon \delta\in\F\}$.
Invoking the transitivity of degenerations (Result~\ref{Result2.2}(i)), together with the fact that $\mathfrak c_3\to\mathfrak l_1$ obtained above, we conclude that $\mathfrak a(\delta)\not\to\mathfrak c_3$ for any $\delta\in\F$.
Finally, $\mathfrak a(\delta)\to\mathfrak c_1$ and $\mathfrak a(\delta)\not\to\mathfrak c_5$ for any $\delta\in\F$, from Remark~\ref{Remark:5.2New}(i) and~(ii) respectively.

The Observations (i)--(iii) above together with Results~\ref{Result2.2} and~\ref{result:2.18} and Remark~\ref{Remark:5.2New}, give sufficient information in order to obtain the complete degeneration picture inside the set $\mathcal C\cap\mathcal N$.
Invoking Observation~(iv), we can thus  complete the degeneration picture inside the set $\mathcal N$.

\begin{center}
\begin{tikzpicture}
\matrix [column sep=7mm, row sep=5mm] {
& \node (n5)  {$\mathfrak c_{5}$};& &\\
&\node (n3)  {$\mathfrak c_{3}$}; & &  \node (n4)  {$\mathfrak a(\delta)$};
 \\
\node (n2)  {$\mathfrak l_1$}; & & \node (n1)  {$\mathfrak c_{1}$};
   \\
 &  \node (n0)   {$\mathfrak a_{0}$}; &
 \\
};
\draw[->, thin] (n5) -- (n3);
\draw[->, thin] (n3) -- (n2);
\draw[->, thin] (n3) -- (n1);
\draw[->, thin] (n4) -- (n1);
\draw[->, thin] (n1) -- (n0);
\draw[->, thin] (n2) -- (n0);
\end{tikzpicture}

{\footnotesize
\refstepcounter{pict}\label{Picture3DimDegenAssocChar=2}
Picture~\ref{Picture3DimDegenAssocChar=2}.
Degenerations of 3-dimensional nilpotent associative algebras over an algebraically closed field $\F$ of ${\rm char}\,\F=2$ (with $\delta\in\F$)
}
\end{center}

\end{document}